\documentclass[a4paper,10pt, number]{elsarticle}

\usepackage[utf8]{inputenc}
\usepackage[total={17cm, 25.7cm}]{geometry}
\usepackage{mathtools, amssymb, mathbbol, braket}
\usepackage[hidelinks]{hyperref}
\usepackage{graphicx}
\usepackage{wasysym}
\usepackage{graphicx, subcaption}

\newtheorem{theorem}{Theorem}
\newtheorem{proposition}[theorem]{Proposition}
\newtheorem{lemma}[theorem]{Lemma}
\newtheorem{corollary}[theorem]{Corollary}
\newtheorem{counterexample}[theorem]{Counterexample}
\newdefinition{notation}[theorem]{Notation}
\newproof{proof}{Proof}

\newcommand{\coloneqq}{:=}
\newcommand{\R}{\mathbb{R}}
\newcommand{\Z}{\mathbb{Z}}

\graphicspath{{./Figs/}}

\begin{document}
	\begin{frontmatter}
		\title{Exact Expressions and Reduced Linear Programmes for the Ollivier Curvature in Graphs}
		\author{Christy Kelly\fnref{fn1}}
		\ead{ckk1@hw.ac.uk}
		\address{Heriot-Watt University, Edinburgh, EH14 4AS}
	\begin{abstract}
		The Ollivier curvature has important applications in discrete geometry and network theory, in particular as a measure of local clustering. It is defined in terms of the Wasserstein distance which, in the discrete setting, can be regarded as an optimal solution of a particular linear programme. In certain classes of graph, this linear programme may be solved \textit{a priori} giving rise to exact combinatorial expressions for the Ollivier curvature. Bhattacharya and Mukherjee (2013) have apparently given exact expressions for the Ollivier curvature in bipartite graphs and graphs of girth 5; we show that these claims do not hold in general and identify the error in the argument of Bhattacharya and Mukherjee. We then repeat the analysis of Bhattacharya and Mukherjee for arbitrary graphs, taking this error into account, and present reduced---parallelly solvable---linear programmes for the calculation of the Ollivier curvature. This allows for potential improvements in the exact numerical evaluation of the Ollivier curvature, though the result heuristically suggests no general exact combinatorial expression for the Ollivier curvature exists. Finally we give an exact expression for the Ollivier curvature in a class of graphs defined by a particular combinatorial constraint motivated by physical considerations. 
	\end{abstract}
	\end{frontmatter}
	
	\section{Introduction}
	The Ollivier curvature \cite{Ollivier_RCMS,Ollivier_RCMCMS,Ollivier_Survey} is a course analogue of the manifold Ricci curvature and is becoming an important tool in discrete geometry \cite{Najman_DiscreteCurvature} and applied network theory \cite{FarooqEtAl_Brain,Ni_RicIntTop,Sandhu_Cancer,Sandhu_Market,SiaEtAl_CommunityDetection,WangEtAl_WirelessNetwork,WangEtAl_Interference,WangEtAl_QUBO,WangEtAl_DiffGeom,WhiddenMatsen_SubtreeGraph}. The author's own interest in the topic relates to its applications in a possible model of quantum gravity in physics \cite{KellyEtAl, KellyTrugenberger,Trugenberger_RHLW, Trugenberger_CombQG}. The Ollivier curvature's foundations are in optimal transport theory and metric measure geometry \cite{Gromov_MetricStructures, Villani_OptimalTransport}, but for the discrete case it (or more precisely the closely related \textit{Wasserstein distance}) may be most readily understood as the solution to a linear programme \cite{Schrijver_LinearProgramming}; while linear programmes can be solved---both in principle and in practice---in polynomial time, the specification of the Ollivier curvature in terms of an optimisation problem makes it a relatively costly quantity to calculate for arbitrary graphs \cite{SamalEtAl}. 
	
	Indeed, recently there has been a move towards other notions of discrete curvature for applications in network theory; in particular, the \textit{Forman curvature} \cite{Forman,SreejithEtAl_Forman3,SreejithEtAl_Forman2,SreejithEtAl_Forman1} is an alternative notion of discrete curvature which is more readily calculable than the Ollivier curvature and typically highly correlated to it \cite{SamalEtAl,WeberEtAl_Forman}. It should be stressed, however, that the underlying intuition of the Ollivier and Forman curvatures are quite different: the former generalises the idea that for two sufficiently nearby open balls in a positively curved Riemannian manifold, the average distance between the two balls is smaller than the distance between their centres; the local coupling between metric and measure theoretic structure is paramount in this conceptualisation. Conversely, the Forman curvature is defined in order to admit a generalisation of the Bochner method and applies to a special class of CW-complexes; it is thus more attuned to global (topological) properties of the space. This has concrete consequences for networks. In particular, for unweighted graphs $G$, the Forman curvature of an edge $uv\in E(G)$ is simply
	\begin{align}
	F(uv) &=-2d_ud_v\left(1-\frac{1}{d_u}-\frac{1}{d_v}\right)
	\end{align}
	where $d_u$ and $d_v$ are the degrees of $u$ and $v$ respectively. Whenever $d_u,\:d_v\geq 2$ this is (up to the factor of $d_ud_v$) the Ollivier curvature of the edge $uv$ under the assumption that the graph $G$ is locally tree-like, i.e. has girth at least $6$ \cite{ChoPaeng_RicCurvCol,JostLiu_RicciCurv}. The Forman curvature thus typically fails to capture any information about local clustering in the network, one of the key properties that the Ollivier curvature measures well \cite{JostLiu_RicciCurv,KellyEtAl, SamalEtAl,  SiaEtAl_CommunityDetection}.
	
	Several classes of (locally finite, simple, unweighted) graph \cite{BhattacharyaMukherjee, ChoPaeng_RicCurvCol, JostLiu_RicciCurv, KellyEtAl, LoiselRomon_RicciCurvPolySurf} admit \textit{a priori} combinatorial expressions for the Ollivier curvature, allowing for analytical analysis and improved computational efficiency; in particular, perhaps the most general results are claimed by Bhattacharya and Mukherjee \cite{BhattacharyaMukherjee} who present expressions for the Ollivier curvature in arbitrary bipartite graphs and graphs of girth at least 5. We provide counterexamples to these claims in section \ref{section: Counterexample} below and identify the oversight in the argument of Bhattacharya and Mukherjee that allows these counterexamples to exist. Despite this oversight, the methods employed by Bhattacharya and Mukherjee do allow for some insight to be gained about the Ollivier curvature in general; in particular we show that in any graph the Ollivier curvature can be obtained by (parallelly) solving a family of reduced linear programmes which in principle might lead to improved numerical evaluation times. From the perspective of exact analysis, however, this result gives strong heuristic indications that no general combinatorial expression for the Ollivier curvature in arbitrary graphs exists, since each of the reduced linear programmes can be arbitrarily complex. It is thus desirable to introduce constraints under which the reduced linear programmes may be entirely solved \textit{a priori}; we identify such a condition and consequently obtain a minor generalisation of the expression given in Ref. \cite{KellyEtAl}; this is of potential physical interest in relation to the introduction of matter to the quantum gravity model of the preceding reference.
	
	Before continuing we shall introduce the notation we shall use for the rest of this text:
	\begin{notation}
		$G$ will denote an arbitrary graph and $uv\in E(G)$ will be an edge of $G$. $N(u)$ and $N(v)$ will denote the sets of neighbours of $u$ and $v$ respectively while $d_u\coloneqq |N(u)|$ and $d_v\coloneqq |N(v)|$ respectively. This notation generalises to arbitrary vertices of $G$. It will also be convenient to define $m_w\coloneqq 1/d_w$, $w\in \set{u,v}$. We now assume $(a,b)\in \set{(u,v),(v,u)}$ for the rest of this definition:
		\begin{enumerate}
			\item $N^v(u)\coloneqq N(u)/\set{v}$ and $N^u(v)\coloneqq N(v)/\set{u}$.
			\item $\triangle(uv)\coloneqq N(u)\cap N(v)$. $\triangle_{uv}\coloneqq |\triangle(uv)|$.
			\item $\square(a)\subseteq N^b(a)$ is the set of neighbours of $a$ that lie on a square supported by $uv$ but not on a triangle supported by $uv$. We let $\square_a\coloneqq |\square(a)|$.
			\item $\pentagon(a)\subseteq N^b(a)$ denotes the set of neighbours of $a$ that lie on a pentagon supported by $uv$ but not on either a square or a triangle supported by $uv$. Also $\pentagon(u,v)\coloneqq \set{a\in G: \rho(a,u)=2\text{ and }\rho(a,v)=2}$. Note that both elements of $\pentagon(u)$ and elements of $\pentagon(v)$ necessarily neighbour elements of $\pentagon(u,v)$. We let $\pentagon_a\coloneqq |\pentagon(a)|$.
			\item We define $\text{Fr}(a)\coloneqq N^b(a)/(\triangle(uv)\cup \square(a)\cup \pentagon(a))$. These are the neighbours of $a$ that do not lie on any short ($3$, $4$ or $5$) cycles supported by $uv$. We denote $n_a\coloneqq |\text{Fr}(a)|$.
			\item $R\coloneqq \triangle(uv)\cup \square(u)\cup \square(v)\cup \pentagon(u)\cup \pentagon(v)\cup \pentagon(u,v)$. The induced subgraph of $G$ specified by $R$ will consist of $K$ connected components $R_1,...,R_K$. We let $\triangle^k(u,v)\coloneqq R_k\cap \triangle(u,v)$, $\square^k(a)\coloneqq  R_k\cap \square(a)$ and $\pentagon^k(a)\coloneqq R_k\cap \pentagon(a)$ for every $k\in \set{1,...,K}$. Then $\triangle_{uv}^k\coloneqq |\triangle^k(u,v)|$, $\square^k_a\coloneqq |\square^k(a)|$ and $\pentagon^k_a\coloneqq |\pentagon^k(a)|$.
		\end{enumerate}
		The \textit{core neighbourhood} of $uv$ is defined as the set $C(u,v)=N(u)\cup N(v)\cup \pentagon(u,v)$; the significance of this set will become apparent shortly. We also write $\alpha\lor \beta\coloneqq \max(\alpha,\beta)$, $\alpha\land \beta\coloneqq \min(\alpha,\beta)$ and $[\alpha]_+\coloneqq \alpha\lor 0$ for all $\alpha,\:\beta\in \R$. For any $\alpha,\:\beta\in \R$, $\theta_{\alpha,\beta}$ is defined such that $\theta_{\alpha,\beta}=1$ iff $\alpha\geq \beta$ and $0$ otherwise. Finally for any vertices $u,\:v\in G$ we let
		\begin{align}
		u\lor v\coloneqq \left\{\begin{array}{rl}
		u, & m_u\geq m_v\\
		v, & m_v>m_u
		\end{array}\right. && u\land v\coloneqq \left\{\begin{array}{rl}
		u, & m_u\leq m_v\\
		v, & m_v<m_u
		\end{array}\right..
		\end{align}
	\end{notation}
	\section{The Method of Bhattacharya and Mukherjee and its Limitations}\label{section: Counterexample}
	Bhattacharya and Mukherjee claim to have derived exact combinatorial expressions for the Ollivier curvature of any edge in a graph $G$ whenever $G$ is either bipartite or has girth greater than $4$ \cite{BhattacharyaMukherjee}. To assess their claims we shall briefly introduce the Ollivier curvature; our presentation is very mercenary and the reader is directed to Refs. \cite{Ollivier_RCMCMS,JostLiu_RicciCurv,BhattacharyaMukherjee,Villani_OptimalTransport} for a more complete introduction to these ideas (especially in the discrete context). Let $G$ be a (simple, locally finite, unweighted) graph equipped with its uniform random walk and let $u$ and $v$ be two vertices of $G$; then (in the Kantorovitch dual formation) the \textit{Wasserstein distance} between $u$ and $v$ is given:
	\begin{align}\label{equation: WassersteinDistance}
	W(u,v)\coloneqq \sup_{x\in \mathcal{L}(G,\R)}W_{uv}(x) && W_{uv}(x)\coloneqq \left(m_u\sum_{a\in N(u)}x(a)-m_v\sum_{a\in N(v)}x(a)\right)
	\end{align}
	where the quantity $W_{uv}(x)$ is defined for every map $x:G\rightarrow \R$ and is called the \textit{profit} or \textit{transport profit} of $x$. $\mathcal{L}(G,\R)$ is the set of all Lipschitz continuous maps $x:G\rightarrow \R$, i.e. maps such that $|x(a)-x(b)|\leq \rho(a,b)$ for any $a,\:b\in G$, where $\rho$ is the standard geodesic graph distance. The Ollivier curvature is then simply defined:
	\begin{align}
	\kappa(u,v)&\coloneqq 1-\frac{W(u,v)}{\rho(u,v)}.
	\end{align}
	While this definition is valid for arbitrary pairs of points, we shall henceforth specialise to the case where $uv$ form an edge of $G$.
	
	We note that the Ollivier curvature of an edge $uv\in E(G)$ is entirely determined in the core neighbourhood $C(u,v)$ of that edge. In particular, by construction $C(u,v)$ essentially consists of the $3$, $4$ and $5$-cycles supported on $uv$ as well as any remaining neighbours of $u$ and $v$. To see this first note that $W_{uv}(x)=W_{uv}(x')$ for any $x,\:x':G\rightarrow \R$ such that $x|N(u)\cup N(v)=x'|N(u)\cup N(v)$ trivially by equation \ref{equation: WassersteinDistance}. Moreover, since for any $(a,b)\in N^v(u)\times N^u(v)$  we have a $3$-path $auvb$, the constraints are entirely preserved if we consider any subgraph which contains all $2$-paths between elements of $N(u)$ and $N(v)$. This implies that it is sufficient to optimise over Lipschitz maps $x:C(u,v)\rightarrow \R$. Note that we are implicitly assuming that every Lipschitz continuous map over $\R$ extends to Lipschitz continuous map on $G$; this assumption has been demonstrated explicitly (c.f. lemma 2.1 of Ref. \cite{BhattacharyaMukherjee}) but is in fact superfluous if the Ollivier curvature is defined in terms of couplings as would be natural in a more complete presentation \cite{KellyEtAl}.
	
	We now turn to the results of Bhattacharya and Mukherjee. Their claims may be summarised as follows:
	\begin{enumerate}
		\item (Theorem 3.1) Let $G$ be a bipartite graph and let $uv$ be an edge of $G$. Theorem 3.1 of Ref. \cite{BhattacharyaMukherjee} claims that  $W(u,v)=W_{BM}(u,v)$ where we define:
		\begin{align}\label{equation: Theorem3.1}
		W_{BM}(u,v)\coloneqq 1+2\left[1-m_u-m_v-m_u\square_u-m_v\square_v +\sum_{k=1}^K(m_u\square^k_u)\lor(m_v\square^k_v)\right]_+.
		\end{align}
		\item (Theorem 3.3) Let $G$ be a graph with girth greater than $4$. Then for every edge $uv\in E(G)$, we define:
		\begin{align}
		W_{BM}(u,v)=1+\left[1-m_u-m_v\right]_++\left[1-m_u-m_v-m_u\pentagon_u-m_v\pentagon_v +\sum_{k=1}^K(m_u\pentagon^k_u)\lor(m_v\pentagon^k_v)\right]_+.
		\end{align}
		Theorem 3.3 of Ref. \cite{BhattacharyaMukherjee} suggests that $W_BM(u,v)=W(u,v)$.
	\end{enumerate}
	\begin{figure}
		\centering
			\begin{subfigure}{0.45\textwidth}
				\centering
				\includegraphics[height=0.45\textwidth]{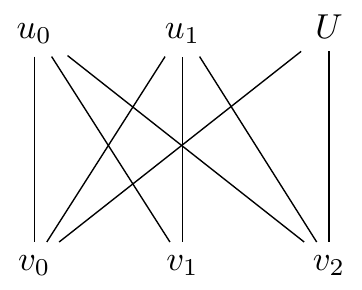}
				\subcaption{Counterexample to Theorem 3.1.}\label{figure1a}
			\end{subfigure}
			\begin{subfigure}{0.45\textwidth}
				\centering
				\includegraphics[height=0.45\textwidth]{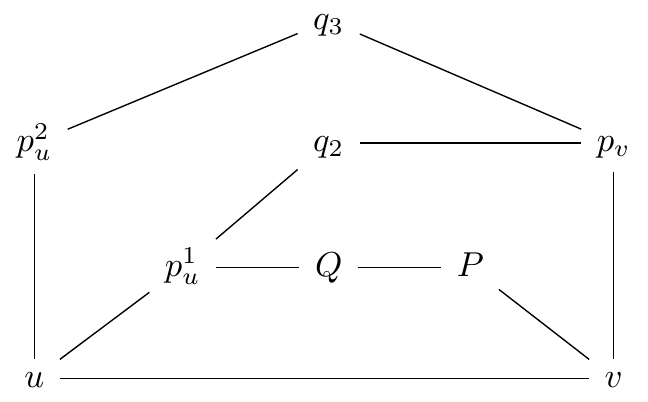}
				\subcaption{Counterexample to Theorem 3.3.}\label{figure1b}
			\end{subfigure}
			\caption{Counterexamples to theorems of Bhattacharya and Mukherjee \cite{BhattacharyaMukherjee} on exact expressions for the Ollivier curvature.}
	\end{figure}
	To demonstrate that $W\neq W_{BM}$ it is sufficient to give a graph $G$ and a Lipshitz function $x:G\rightarrow \R$ such that $W_{uv}(x)>W_{BM}(u,v)$.
	\begin{counterexample}\label{counterexample1}[Theorem 3.1]
		Consider the bipartite graph in figure \ref{figure1a}, where $U$ is a family of vertices such that $N(u)=\set{v_0,v_2}$ for each $u\in U$. By equation \ref{equation: Theorem3.1} we have $W_{BM}(u_0,v_0)=1$ for any nonvoid set $U$; however the mapping
		\begin{align}
		x(a)&=\left\{\begin{array}{rl}
		1, & a = v_1,\\
		0, & a\in \set{u_0,u_1}\\
		-1, & a\in \set{v_0,v_2}\\
		-2, & a\in U
		\end{array}\right.
		\end{align}
		is a Lipschitz function and has profit 
		\begin{align}
		W_{u_0v_0}(x)=\frac{5|U|-2}{3|U|+6}
		\end{align}
		and $W_{u_0v_0}(x)>W_{BM}(u_0,v_0)$ for any $|U|>4$.
	\end{counterexample}
	\begin{proof}
		The graph manifestly has the bipartition $V_1=\set{u_0,u_1}\cup U$ and $V_2=\set{v_0,v_1,v_2}$. Clearly $d_{u_0}=3$ and $d_{v_0}=|U|+2$ while $\square(u_0)=\set{v_1,v_2}$ and $\square(v_0)=\set{u_1}\cup U$ while $R$ is a single connected component since $\set{u_1}\cup U\in N(v_2)$ and $v_1,\:v_2\in N(u_1)$. Given this we have
		\begin{align}
		W_{BM}(u_0,v_0)&=1+2\left[1-\frac{1+2}{3}-\frac{1+(|U|+1)}{|U|+2}+\left(\frac{2}{3}\right)\lor\left(\frac{|U|+1}{|U|+2}\right)\right]_+=1+2\left[\left(\frac{2}{3}\right)\lor\left(\frac{|U|+1}{|U|+2}\right)-1\right]_+=1\nonumber.
		\end{align}
		To note that $x$ is Lipschitz continuous it is sufficient to recognise that $U\cap( N(u_0)\cup N(u_1)\cup N(v_1))=\emptyset$ and $\set{v_0,v_2}\cap N(v_1)=\emptyset$. Then $x$ has the transport profit
		\begin{align}
		W_{u_0v_0}(x)&=\frac{1}{3}(-1+1-1)-\frac{1}{|U|+2}(0+0-2|U|)=\frac{2|U|}{|U|+2}-\frac{1}{3}=\frac{5|U|-2}{3|U|+6}\nonumber
		\end{align}
		as required. It is then easily verified that $(5|U|-2)/(3|U|+6)>1\text{ iff }|U|>4$ and so any such graph provides a counterexample to the expression from Bhattacharya and Mukherjee.
	\end{proof}
	\begin{counterexample}[Theorem 3.3]
		Consider the graph in figure \ref{figure1b}, where $Q$ and $P$ denote families of vertices such that each $q\in Q$ is incident two exactly two edges $p_u^1q$ and $qp_q$ with $p_q\in P$, and each $p\in P$ is also incident to exactly two edges $q_pp$ and $pv$ where $q_p\in Q$. If $|P|=6$, equation \ref{equation: Theorem3.1} gives $W_{BM}(u_0,v_0)=37/24$; however the mapping
		\begin{align}
		x(a)&=\left\{\begin{array}{rl}
		1, & a=p_u^2\\
		0, & a\in \set{u,p_u^1,q_3}\\
		-1, & a\in \set{v,p_v,q_2}\cup Q\\
		-2, & a\in P
		\end{array}\right.
		\end{align}
		is a Lipschitz function and has profit $13/8>37/24$, i.e. $W_{uv}(x)>W_{BM}(u,v)$ which contradicts the claim that $W_{BM}=W$.
	\end{counterexample}
	\begin{proof}
		The graph is manifestly of girth $5$ by construction, while $\pentagon_u=2$, $\pentagon_v=|P|+1=7$, $d_u=3$ and $d_v=8$. Then:
		\begin{align}
		W_{BM}(u,v)&=1+\left[1-\frac{1}{3}-\frac{1}{8}\right]_++\left[1-\frac{1}{3}-\frac{1}{8}-\frac{2}{3}-\frac{7}{8}+\left(\frac{2}{3}\right)\lor\left(\frac{7}{8}\right)\right]_+\nonumber\\
		&=\frac{37}{24}\nonumber.
		\end{align}
		$x$ is a Lipschitz function because 
		\begin{align}
		\rho(p_u^2,v)=\rho(p_u^2,p_v)=\rho(u,p)=\rho(p_u^1,p)=2 && \rho(p_u^2,q_3)=\rho(p_u^2,q)=\rho(q_3,p)=3\nonumber
		\end{align}
		for all $p\in P$ and $q\in Q$. It has the transport profit
		\begin{align}
		W_{uv}(x)&=\frac{1}{3}(1+0-1)-\frac{1}{8}(0-1-2|P|)=\frac{13}{8}=\frac{39}{24}\nonumber
		\end{align}
		which is greater than $W_{BM}(u,v)$ as required.
	\end{proof}
	
	The essence of the method of Bhattacharya and Mukherjee is to utilise the fact that the Wasserstein distance is defined by a linear programme with \textit{integral} solutions. We first reformulate the Wasserstein distance in terms of linear optimisation theory \cite{BhattacharyaMukherjee}: let $n=|N(u)\cup N(v)|$, and let us index the standard basis of $\R^n$ by $N(u)\cup N(v)$ ordered in the block form
	\begin{align}
	[u,v,\triangle(u,v),N(u)/\triangle(u,v),N(v)/\triangle(u,v)]\nonumber.
	\end{align}
	Then we may identify vectors $\textbf{x}=(x_a)\in \R^n$ with mappings $x:G\rightarrow \R$ such that $x_a=x(a)$ for each $a\in N(u)\cup N(v)$. The graph $G$ defines a symmetric digraph $D(G)$ in which any edge $ab$ of $G$ specifies two directed edges $(a,b)$ and $(b,a)$ in $D(G)$; the incidence matrix of a digraph with $E$ edges and $N$ vertices is an $E\times N$ dimensional matrix such that the entry $(e,v)$ takes on the value $1$ if the edge $e$ enters the vertex $v$, $-1$ if it leaves $v$ and $0$ if it is not incident with $v$. Thus the Lipschitz continuous criterion on $x:G\rightarrow \R$ may be reformulated in terms of the following \textit{linear} constraint:
	\begin{align}\label{equation: LinearConstrain}
	M\textbf{x}\leq \mathbb{1}
	\end{align}
	where $\mathbb{1}$ is the $2n$-dimensional column vector with entries all equal to $1$. Now defining
	\begin{align}\label{equation: CostFunction}
	\textbf{w}\coloneqq \begin{bmatrix}
	m_v, & m_u, &	m_u-m_v, & m_u, & m_v
	\end{bmatrix}^T
	\end{align}
	immediately ensures that $\textbf{w}^T\textbf{x}=W_{uv}(x)$ so we have a linear programme in the canonical form; note that $W_{uv}(x)=W_{uv}(x+c)$ for any constant $c\in \R$ so we may assume that $0\leq \textbf{x}$ without loss of generality as long as the programme is bounded (as is the case).
	
	We can now rewrite lemma 2.2 of Ref. \cite{BhattacharyaMukherjee}:
	\begin{theorem}
		The linear programme defined by $\max d_ud_v\textbf{w}^T\textbf{x}$ where $\textbf{w}$ is given by equation \ref{equation: CostFunction} subject to the constraints \ref{equation: LinearConstrain} and $x_u=0$ is integral, i.e. we may assume that an optimal solution $\textbf{x}\in \Z^n$ without loss of generality. 
	\end{theorem}
	\begin{proof}
		The proof uses some standard ideas of integer linear programming: the incidence matrix of every digraph is \textit{totally unimodular} while every linear programme with a totally unimodular constraint matrix (and integral inputs) is integral \cite{Schrijver_LinearProgramming}. Note that the constraint $x_u=0$ follows from the fact that $W_{uv}(x)=W_{uv}(x+c)$ for all $c\in \R$ as mentioned previously.
	\end{proof}
	\begin{corollary}\label{corollary: WassDist}
		Fix an edge $uv\in E(G)$ and let $\mathcal{L}_{uv}^\alpha\coloneqq \set{x\in \mathcal{L}(G,\R):x(u)=0\text{ and }x(v)=\alpha}$. If we let $W_{\alpha}(u,v)\coloneqq \sup_{x\in \mathcal{L}_{uv}^\alpha}W_{uv}(x)$ for each $\alpha\in \set{-1,0,1}$, then $W(u,v)=W_{-1}(u,v)\lor W_0(u,v)\lor W_{1}(u,v).$ 
	\end{corollary}
	It is thus sufficient to maximise over each $W_\alpha$, $\alpha\in \set{-1,0,1}$, separately and take the maximum. Progress can be made in this direction if we note that by restricting to the calculation $W_\alpha$, we impose stricter bounds on the values of $x(a)$, $a\in N(u)\cup N(v)/\set{u,v}$. These lead to bounds on the values of $W_{uv}(x)$, $x\in \mathcal{L}_{uv}^\alpha$, so if a Lipschitz continuous map can be found that realises this bound then linear programme has been solved. For instance:
	\begin{proposition}\label{proposition: W1}
		$W_1(u,v)=1-\triangle_{uv}m_u\land m_v$ for any edge $uv\in E(G)$.
	\end{proposition}
	\begin{proof}
		We wish to derive the \textit{a priori} inequalities on elements of $x(a)$, $a\in N(u)\cup N(v)$, that arise from Lipschitz continuity. We may assume $x(u)=0$ without loss of generality so $-1\leq f(a)\leq 1$ for all $a\in N(u)$ and all $x\in \mathcal{L}^\alpha_{uv}$. We are also only interested in the case $x(v)=1$ i.e. $0\leq x(b)\leq 2$ for all $b\in N(v)$. However we must also note that if $a\in \triangle(uv)$ we have the improved bounds $0\leq f(a)\leq 1$. Thus if we write
		\begin{align}
		W_{uv}(x)&=m_u+m_u\sum_{a\in N^v(u)/\triangle(uv)}x(a)-m_v\sum_{b\in N^u(v)/\triangle(uv)}x(b)+\left(m_u-m_v\right)\sum_{a\in \triangle(uv)}x(a)\nonumber
		\end{align}
		we have
		\begin{align}
		W_{uv}(x)&\leq m_u\left(1+(d_u-1-\triangle_{uv})\right)+\left[\triangle_{uv}\left(m_u-m_v\right)\right]_+=1-m_u\land m_v\triangle_{uv}\nonumber
		\end{align}
		where we have used the fact that $-m_u+\left[m_u-m_v\right]_+$ is equal to $-m_u$ if $d_u\geq d_v$ and is equal to $-m_v$ otherwise. The RHS is the transport profit of the map
		\begin{align}
		x(a)&=\left\{\begin{array}{rl}
		1, & a\in N(u)/\triangle (uv)\\
		0, & a\in N(v)/\triangle(uv)\\
		\theta_{m_u,m_v}, & a\in \triangle(uv)
		\end{array}\right.\nonumber.
		\end{align}
		This is clearly well-defined; it is a Lipschitz function trivially for all graphs because $\text{Im}_x(N(u)\cup N(v))=\set{0,1}$, i.e. $|x(a)-x(b)|\leq 1$ for any $a,\:b\in N(u)\cup N(v)$ while $\rho(a,b)\geq 1$ for all distinct $u,\:v\in N(u)\cup N(v)$.
	\end{proof}
	Unfortunately, for $\alpha\neq 1$ it does not seem to be possible to derive results of similar generality, as we shall argue in the next section. Bhattacharya and Mukherjee argue that by restricting to a single type of cycle per connected component $R_k$ of $R$, it is possible to give a recursive construction of the optimal mapping $x:G\rightarrow \R$. We illustrate the problem for the case $\alpha=0$ when $G$ is bipartite but the basic error is the same in all other cases. For simplicity we also assume that $R$ consists of a single connected component and that $\text{Fr}(u)=\text{Fr}(v)=\emptyset$. Given these assumptions we have
	\begin{align}
	W_{uv}(x)=m_u\sum_{a\in \square(u)}x(a)-m_v\sum_{b\in \square(v)}x(b).
	\end{align}
	Given that $a\in N(u)$ and $b\in N(v)$, we have the bounds $-1\leq a,\:b\leq 1$ and so naively we may maximise $W_{uv}(x)$ by choosing $x(a)=1$ and $x(b)=-1$ for all $a\in N(u)$ and all $b\in N(v)$. This however is not Lipshitz continuous since some $x\in \square(u)$ neighbours some $y\in \square(v)$ and $|x(a)-x(b)|=2>\rho(a,b)=1$. As such Lipschitz continuity requires that the two terms are maximised in conjunction. Bhattacharya and Mukherjee assume that an optimal mapping $x:G\rightarrow \R$ is given by a recursive construction which maintains the maximum value at each step: in particular let $a_0\in N(u)\cup N(v)$ and set $x(a_0)$ to the value that maximises $W_{uv}(x)$ for the range of values possible for $a_0$. At the next step assign the maximising values to each of the neighbours of $a_0$ that ensure Lipschitz continuity etc. In this way one ends up either with a mapping $x$ that assigns $1$ to each $R_k\cap N(u)$ and $0$ to each $R_k\cap N(v)$ or alternatively that assigns $0$ to each $R_k\cap N(u)$ and $-1$ to each $R_k\cap N(v)$ (one can ignore problems caused by triangles in this assignment since the graph is bipartite); the two distinct cases arise depending on whether $a_0\in N(u)$ or $a_0\in N(v)$. The counterexamples presented above essentially show that this construction is not in fact optimal under the assumptions made in Ref. \cite{BhattacharyaMukherjee}.
	\section{Linear Programmes for Ollivier Curvature Evaluation}\label{section: LinearProgramme}
	In this section we show that the problem of calculating the Wasserstein distance reduces to the problem of solving a pair of linear programmes for each connected component $R_k$ of $R$; however, in general, $R_k$ can be \textit{any graph} while all the available \textit{a priori} information has already been used in the reduction of the linear programme to this restricted form and it seems likely that there is no combinatorial expression for the Ollivier curvature in arbitrary graphs. Making this statement precise would require a precise notion of a combinatorial expression for the Ollivier curvature.
	
	We begin with two elementary lemmas:
	\begin{lemma}\label{lemma: W0}
		Let $uv$ be an edge in $G$. 
		\begin{enumerate}
			\item For any $x\in \mathcal{L}_{uv}^0$, there is a $y\in \mathcal{L}_{uv}^0$ given by
			\begin{align}
			y(a)&=\left\{\begin{array}{rl}
			1, & a\in \pentagon(u)\cup\text{Fr}(u)\\
			0, & a\in \set{u,v}\cup \pentagon(u,v)\\
			-1, & a\in \pentagon(v)\cup \text{Fr}(v)\\
			f(a), & a\in \triangle(uv)\cup \square(u)\cup \square(v)
			\end{array}\right.
			\end{align}
			such that $W_{uv}(y)\geq W_{uv}(x)$.
			\item Given any mapping $x\in \mathcal{L}_{uv}^0$ such that $x(a)=-1$ for some $a\in \square(u)$, the mapping $y:C(u,v)\rightarrow \Z$ given by $y(b)=x(b)$, $b\neq a$, and $y(a)=0$ is Lipshitz continuous and satisfies $W_{uv}(y)\geq W_{uv}(x)$. Similarly, for any Lipschitz continuous map $x:C(u,v)\rightarrow \Z$ with $x(u)=x(v)=0$ and $x(a)=1$ for some $a\in \square(v)$, there is a mapping $y\in \mathcal{L}_{uv}^0$ such that $W_{uv}(y)\geq W_{uv}(x)$ given by $y(a)=0$ and $y(b)=x(b)$ for all $b\neq a$.
		\end{enumerate} 
	\end{lemma}
	\begin{proof}
		\leavevmode
		\begin{enumerate}
			\item It is easy to see that $W_{uv}(y)\geq W_{uv}(x)$:
			\begin{align}
			W_{uv}(y)-W_{uv}(x)&=m_u\sum_{a\in \pentagon(u)\cup\text{Fr}(u)}(1-x(a))+m_v\sum_{b\in \pentagon(v)\cup \text{Fr}(v)}(1+x(b))\nonumber.
			\end{align}
			Since $-1\leq x(a),\:x(b)\leq 1$ for all $a\in N(u)$ and $b\in N(v)$, $(1-x(a))\geq 0$ and $(1+x(b))\geq 0$ for all $a\in N(u)$ and all $b\in N(v)$ and $W_{uv}(y)-W_{uv}(x)\geq 0$. It thus remains to show that $y$ is Lipshitz continuous. This follows if we notice that the (block form) matrix
			\begin{align}\label{equation: DistMat}
			\mathbb{D}\coloneqq \begin{array}{c|cccccc}
			& v & \triangle(u,v) & \square(u) & \pentagon(u) & \text{Fr}(u)  & \pentagon (u,v) \\ \hline 
			u & 1 & 1 & 1 & 1 & 1 & 2\\
			\triangle(u,v) & 1 & 1 & 1 & 2 & 2 & 1\\
			\square(v) & 1 & 1 & 1 & 2 & 3 & 1\\
			\pentagon(v) & 1 & 2 & 2 & 2 & 3 & 1\\
			\text{Fr}(v) & 1 & 2 & 3 & 3 & 3 & 2\\
			\pentagon(u,v) & 2 & 1 & 1 & 1 & 2 & 1
			\end{array}
			\end{align}
			has as the block entry $(i,j)$ the least possible distance between an element of the set labelling row $i$ and a distinct element of the set labelling the column $j$. We may index $\mathbb{D}$ by elements $(a,b)$ rather than blocks $(i,j)$ in which case $\mathbb{D}_{xy}=\mathbb{D}_{ij}$ where $i$ corresponds to the equivalence class of $a$ in the core neighbourhood partition and $j$ corresponds to the equivalence class of $b$ in the core neighbourhood partition. Letting $A\coloneqq \triangle(u,v)\cup \square(u)\cup \square(v)$, we note that if $a,\:b\in A$ then $|y(a)-y(b)|=|x(a)-x(b)|\leq \rho(x,y)$ since $x$ is Lipshitz continuous. If $a\in A$ and $b\notin A$, we have $|y(a)-y(b)|=|x(a)-y(b)|\leq 2\leq \mathbb{D}_{ab}\leq \rho(a,b)$, where the second step follows because of generic bounds on $x(a)$ and $y(b)$, $a\in N(u)$, $b\in N(v)$, and the final step is valid because each entry of $\mathbb{D}$ is the \textit{least} distance between the blocks associated to $a$ and $b$ respectively. By the same arguments we have $|y(a)-y(b)|\leq 2 \leq \mathbb{D}_{xy}\leq \rho(a,b)\nonumber$ for $x,\:y\notin A$ as required.
			\item We consider a mapping mapping $x\in \mathcal{L}_{uv}^0$ such that $x(a)=-1$ for some $a\in \square(u)$. To see that $y$ is Lipschitz continuous it is sufficient to note that since $x$ is Lipschitz continuous and $x(a)=-1$, $x(b)\in \set{0,-1}$ for all $b\in N(a)\cap C(u,v)$ and so $|y(a)-x(b)|=|x(b)|\leq 1=\rho(a,b)$ as required. It is obvious that $W_{uv}(y)-W_{uv}(x)=m_u(0-(-1))=m_u>0$ as required. \textit{Mutatis mutandis}, the same argument proves the statement for $x\in \mathcal{L}_{uv}^{0}$ such that $x(a)=1$ for some $a\in \square(v)$.
		\end{enumerate}
	\end{proof}
	\begin{lemma}\label{lemma: W-1}
		Let $uv$ be an edge in $G$. and consider $x\in \mathcal{L}_{uv}^{-1}$.
		\begin{enumerate}
			\item The mapping
			\begin{align}
			y(a)&=\left\{\begin{array}{rl}
			1, & a\in \text{Fr}(u)\\
			0, & a = u\\
			-1, & a=v\\
			-2, & a\in \text{Fr}(v)\\
			x(a), & a\in R
			\end{array}\right.
			\end{align}
			is Lipschitz continuous and has $W_{uv}(y)\geq W_{uv}(x)$.
			\item Suppose that there is an $a\in \square(u)$ such that $x(a)=-2$. Then the mapping $y:C(u,v)\rightarrow \Z$ defined by $y(a)=-1$ and $y(b)=x(b)$ for $a\neq b$ is Lipschitz continuous and satisfies $W_{uv}(y)\geq W_{uv}(x)$. Similarly if there is an $a\in \square(v)$ such that $x(a)=1$; then the mapping $y:C(u,v)\rightarrow \Z$ defined by $y(a)=0$ and $y(b)=y(b)$ for $a\neq b$ belongs to $\mathcal{L}_{uv}^{-1}$ and satisfies $W_{uv}(y)\geq W_{uv}(x)$.
			\item Let $A\coloneqq \set{a\in \pentagon(u):x(a)\in \set{-1,-2}}$. Then there is a Lipschitz continuous mapping $y:C(u,v)\rightarrow \Z$ such that $y(u)=0$, $y(v)=-1$, and $y(a)=0$ for each $a\in A$ and such that $W_{uv}(y)\geq W_{uv}(x)$. Similarly, if $A\coloneqq \set{a\in \pentagon(v):x(a)\in \set{1,0}}$, then there is a mapping $y:C(u,v)\rightarrow \Z$ given by $y(a)=-1$ that belongs to $\mathcal{L}_{uv}^{-1}$ and satisfies $W_{uv}(y)\geq W_{uv}(x)$.
		\end{enumerate}
	\end{lemma}
	\begin{proof}
		\leavevmode
		\begin{enumerate}
			\item To see that the mapping $y$ given above is Lipschitz continuous it will be useful to compare with the matrix \ref{equation: DistMat}. Now note that if $a,\:b\in R\cup\set{u,v}$ then $|y(a)-y(b)|=|x(a)-x(b)|\leq \rho(a,b)$ since $x$ is Lipschitz continuous. Similarly if $a\in R$ and $b\in C(u,v)/(R\cup \set{u,v})=\text{Fr}(u)\cup\text{Fr}(v)$ then $|y(a)-y(b)|=|x(a)-y(b)|$; if $a\in \triangle(u,v)$ then $x(a)\in \set{-1,0}$ so since $-2\leq y(b)\leq 1$ for all $b\in R$ we have $|y(a)-y(b)|\leq 2=\mathbb{D}_{ab}\leq \rho(a,b)$. Similarly if $a\in R/\triangle(u,v)$ then $-2\leq x(a) \leq 1$ and $-2\leq y(b)\leq 1$ so $|y(a)-y(b)|\leq 3=\mathbb{D}_{ab}=\rho(a,b)$. Now we consider $a,\:b\in \text{Fr}(u)\cup \text{Fr}(v)$. Then $|y(a)-y(b)|=3=\rho(a,b)$ and $y$ is Lipshitz continuous as required. Now we note that
			\begin{align}
			W_{uv}(y)-W_{uv}(x)&=m_u\sum_{a\in \text{Fr}(u)}(1-x(a))+m_v\sum_{b\in \text{Fr}(v)}(2+x(b))\nonumber
			\end{align}
			so since $-2\leq x(a),\:x(b)\leq 1$ for all $a,\:b\in R$ we have $1-x(a)\geq 0$ and $2+x(b)\geq 0$, i.e. $W_{uv}(y)\geq W_{uv}(x)$ as required.
			\item If $x(a)=-2$ then each neighbour $b\in N(a)$ satisfies $b\in \set{-2,-1}$ and $y\in \mathcal{L}_{uv}^{-1}$ trivially. Clearly $W_{uv}(y)-W_{uv}(x)=m_u>0$; similarly arguments ensure the second part of this point.
			\item We prove the first statement: because $x$ is Lipschitz continuous, $y$ is also Lipschitz continuous unless there is some $a\in \pentagon(u)$ with a neighbour $b$ such that $y(b)=-2$; but because $a\in \pentagon(u)$, $b\notin N(v)$, i.e. $b\in \pentagon(u,v)\cup N(u)/\triangle(u,v)$. If $b\in \pentagon(u,v)$ we may assume that $y(b)=-1$ without loss of generality since this ensures that $x$ is Lipschitz continuous and we may reassign the value of points in $\pentagon(u,v)$ without changing the transport profit. Similarly by part (ii) above we may assume that if $b\in \square(u)$ then $x(b)\geq -1$; thus an obstruction $b$ to assuming $y(a)=0$ requires $b\in \pentagon(u)$ and $y(b)=-2$; but by construction no such $b$ exists since we set $y(b)=0$ for any $b\in \pentagon(u)$ such that $x(u)=-2$ and $y$ is a Lipschitz function. Also $W_{uv}(y)\geq W_{uv}(x)$ since the LHS exceeds the right by at least $m_u|A|$. Essentially the same arguments prove the second part of this statement.
		\end{enumerate}
	\end{proof}
	\begin{theorem}\label{theorem: ConnectedComponents}
		Let $G$ be a graph and $uv$ an edge of $G$. For each connected component $R_k$ of $R\subseteq C(u,v)$, let
		\begin{align}
		\textbf{c}_k^0\coloneqq [m_u-m_v,m_u,m_v]^T && \textbf{c}_k^{-}\coloneqq [m_u-m_v,m_u,m_u,m_v,m_v]^T,
		\end{align}
		where the former is a vector of $\mathbb{R}^{n_1}$, $n_1\coloneqq (\triangle_{uv}^k+\square^k_u+\square_v^k)$, and the a latter vector of $\mathbb{R}^{n_2}$, $n_2\coloneqq (n_1+\pentagon_u^k+\pentagon_v^k)$, indexed by the blocks $[\triangle^k(u,v),\square^k(u),\square^k(v)]$ and $[\triangle^k(u,v),\square^k(u), \pentagon^k(u),\square^k(v), \pentagon^k(v)]$ respectively. Also let $M_k^0$ and $M_k^{-}$ be the incidence matrices of the digraphs associated to the induced subgraphs of $G$ on the vertex sets $\triangle^k(u,v)\cup\square^k(u)\cup \square^k(v)$ and $R_k$ respectively. Then if $\textbf{x}_k^0$ and $\textbf{x}_k^{-}$ are optimal solutions to the linear programmes that seek to maximise $(\textbf{c}_k^0)^T\textbf{x}^0$ and $(\textbf{c}_k^{-})^T\textbf{x}^{-}$ subject to the constraints
		\begin{subequations}\label{equation: ReducedLinProgrammeConstraints0}
			\begin{align}
			\begin{bmatrix}\label{equation: ReducedLinProgrammeConstraints0A}
			-1 & 0 & -1
			\end{bmatrix}&\leq (\textbf{x}^0)^T\leq \begin{bmatrix}
			1 & 1 & 0
			\end{bmatrix}\\
			-\mathbb{1}&\leq M_k^0\textbf{x}^0\leq \mathbb{1}\label{equation: ReducedLinProgrammeConstraints0B}
			\end{align}
		\end{subequations}
		and 
		\begin{subequations}\label{equation: ReducedLinProgrammeConstraints1}
			\begin{align}
			\begin{bmatrix}\label{equation: ReducedLinProgrammeConstraints1A}
			-1 & -1 & 0 & -2 & -2
			\end{bmatrix}&\leq (\textbf{x}^{-})^T\leq \begin{bmatrix}
			0 & 1 & 1 & 0 & -1
			\end{bmatrix} \\
			-\mathbb{1}&\leq M_k^{-}\textbf{x}^{-}\leq \mathbb{1}\label{equation: ReducedLinProgrammeConstraints1B}
			\end{align}
		\end{subequations}
		respectively, then the Wasserstein distance is given
		\begin{align}\label{equation: WassersteinDistanceConnectedComp}
		W(u,v)&=\left(1-\frac{\triangle_{uv}}{d_u\lor d_v}\right)\lor\left(\frac{n_u+\pentagon_u}{d_u}+\frac{n_v+\pentagon_v}{d_v}+\sum_{k}(\textbf{c}_k^0)^T\textbf{x}^0_k\right)\lor\left(\frac{n_u-1}{d_u}+\frac{2n_v}{d_v}+\sum_{k}(\textbf{c}_k^{-})^T\textbf{x}^{-}_k\right).
		\end{align}
	\end{theorem}
	\begin{proof}
		The expression \ref{equation: WassersteinDistanceConnectedComp} follows immediately from corollary \ref{corollary: WassDist} and proposition \ref{proposition: W1} if we can show that 
		\begin{align}
		W_0(u,v)=\frac{n_u+\pentagon_u}{d_u}+\frac{n_v+\pentagon_v}{d_v}+\sum_{k}(\textbf{c}_k^0)^T\textbf{x}^0_k\nonumber && W_{-1}(u,v)=\frac{n_u-1}{d_u}+\frac{2n_v}{d_v}+\sum_{k}(\textbf{c}_k^{-})^T\textbf{x}^{-}_k\nonumber.
		\end{align}
		But by part (i) of lemmas \ref{lemma: W0} and \ref{lemma: W-1} the only contentious part of these expressions is related to the sum over connected components. These are adequate if we can show that the constraints \ref{equation: ReducedLinProgrammeConstraints0} and \ref{equation: ReducedLinProgrammeConstraints1} imply the constraints \ref{equation: LinearConstrain} and $x_u=0$ for a the general linear programme that specifies the Wasserstein distance $W$. We first consider the constraints \ref{equation: ReducedLinProgrammeConstraints1}; since $R_k$ is a connected component of $R\subseteq C(u,v)$ it contains all neighbours of any element $a\in R_k$ other than $u$ or $v$. Thus for any $a\in R_k$, the constraints \ref{equation: LinearConstrain} are automatically implied by the bounds \ref{equation: ReducedLinProgrammeConstraints1B}, except for those related to edges of the form $ua$ or $va$. Given $x_u=0$, $x_v^0=0$ and $x_v^-=-1$, these edges imply the additional constraints $[-1,-1,-1,-2,-2,]\leq (\textbf{x}^{-})^T\leq [0,1,1,0,0]$, while we can assume that the slightly stricter constraints \ref{equation: ReducedLinProgrammeConstraints1A} hold for optimal solutions by parts (ii) and (iii) of lemma \ref{lemma: W-1}. Similar remarks hold for the condition \ref{equation: ReducedLinProgrammeConstraints0B} except we need to additionally verify that constraints coming from edges of the form $ab$ where $b\in \pentagon^k(u)\cup \pentagon^k(v)$ are satisfied. Concretely, edges of the form $ua$ and $va$ imply the additional constraints $[-1,-1,-1]\leq (\textbf{x}^0)^T\leq [1,1,1]$ and the improved constraints \ref{equation: ReducedLinProgrammeConstraints0A} arise if we apply part (ii) of lemma \ref{lemma: W0}. For edges of the form $ab$ where $b\in \pentagon^k(u)$ we note that if $a\in R_k\cap N(v)$ then $\rho(a,b)\geq 2$ as otherwise the shortest cycle supported by $uv$ incident with $b$ would not be a pentagon. A similar argument holds if $a\in R_k\cap N(u)$ and $b\in \pentagon^k(v)$ so the only potential issues arise when we have an edge $ab$ with $a\in R_k\cap N(w)$ and $b\in \pentagon^k(w)$, $w\in \set{u,v}$. But since $\triangle^k(u,v)=R_k\cap N(u)\cap N(v)$ such edges can only exist if $a\in \square^k(w)\cup \pentagon^k(w)$ in which case the constraints \ref{equation: ReducedLinProgrammeConstraints0} ensure that $|x(a)-x(b)|\leq 1\leq \rho(a,b)$ as required.
	\end{proof}
	Heuristically this theorem suggests that no exact combinatorial expression for the Ollivier curvature exists; this is because a connected component $R_k$ can take the form of any graph $H$ given an appropriate choice of edges to the vertices $u$ and $v$ of $G$ (trivially one can connect all vertices of $H$ to both $u$ and $v$). There has thus been no essential reduction in the difficulty of the problem despite imposing all the obviously available \textit{a priori} information. Nonetheless, by imposing restrictions on the form of the connected components $R_k$, one can obtain new exact results. We will need some new notation:
	\begin{notation}
		We let $\square_\triangle(u)$ and $\square_\triangle(v)$ denote the subsets of $\square(u)$ and $\square(v)$ respectively that neighbour an element of $\triangle(u,v)$. Then $\square_\circ(u)\coloneqq \square(u)/\square_\triangle(u)$ and $\square_\circ(v)\coloneqq \square(u)/\square_\triangle(v)$. Similarly $\pentagon_{\circ}(u)$ and $\pentagon_{\circ}(u)$ are those elements of $\pentagon(u)$ and $\pentagon(v)$ respectively with a $2$-path in $R$ to $\triangle(u,v)$. Then $\square_\triangle^u\coloneqq |\square_\triangle(u)|$, $\square_\triangle^v\coloneqq |\square_\triangle(v)|$, $\square_\circ^u\coloneqq |\square_\circ(u)|$, $\square_\circ^v\coloneqq |\square_\circ(v)|$, $\pentagon_{\circ}^u\coloneqq |\pentagon_{\circ}(u)|$ and $\pentagon_{\circ}^v\coloneqq |\pentagon_{\circ}(v)|$.
	\end{notation}
	\begin{proposition}
		Let $G$ be a graph and $uv$ an edge such that $|R_k\cap (N(u)\cup N(v))|\leq 2$ for each connected component $R_k$ of $R\subseteq C(u,v)$. Then 
		\begin{align}
		\kappa(u,v)&=m_u\land m_v\triangle_{uv}-[\kappa_0(u,v)]_+-[\kappa_{-1}(u,v)]_+\nonumber\\
		\kappa_0(u,v)&\coloneqq 1-m_u-m_v-m_u\land m_v(\triangle_{uv}+ \square_\circ) -m_u\land [m_v-m_u]_+\square_\triangle^u -m_v\land [2m_u-3m_v]_+\square_\triangle^v\nonumber\\
		\kappa_{-1}(u,v)&\coloneqq 1-m_u-m_v-m_u\lor m_v\triangle_{uv}-m_u\land m_v (\square_\circ+\pentagon_{\circ}) -(m_v\land [m_u-m_v]_+-m_v\land [2m_u-3m_v]_+)\square_\triangle^v\nonumber
		\end{align}
		where $\square_\circ\coloneqq \square_\circ^u=\square_\circ^v$ and $\pentagon_{\circ}\coloneqq \pentagon_{\circ}^u=\pentagon_{\circ}^v$.
	\end{proposition}
	\begin{proof}
		We use the notation of theorem \ref{theorem: ConnectedComponents}; $x_k^0$ and $x_k^{-}$ are mappings corresponding to the optimal solutions $\textbf{x}_k^0$ and $\textbf{x}_k^{-}$  respectively. The essential idea is that we shall deduce expressions for $(\textbf{c}_k^\ominus)^T\textbf{x}^\ominus$ for every possible configuration of $R_k$ under the assumptions of the proposition by using the constraints \ref{equation: ReducedLinProgrammeConstraints0B} and \ref{equation: ReducedLinProgrammeConstraints1B} to bound $(\textbf{c}_k^0)^T\textbf{x}^0$ and $(\textbf{c}_k^-)^T\textbf{x}^-$ respectively, and then provide a Lipschitz function realising these bounds. 
		
		Suppose $|R_k\cap (N(u)\cup N(v))|=1$; then the unique vertex $a\in R_k\cap (N(u)\cup N(v))$ is an element of $\triangle^k(u,v)$ of necessity and $(\textbf{c}_k^\ominus)^T\textbf{x}^\ominus=(m_u-m_v)x_k^\ominus(a)$. We thus have the bounds $(\textbf{c}_k^0)^T\textbf{x}^0 \leq \left(m_u\lor m_v-m_u\land m_v\right)$ and $(\textbf{c}_k^{-})^T\textbf{x}^{-}\leq \left[m_v-m_u\right]_+$. Setting $x_k^0(a)=\text{sign}(d_v-d_u)$ and $x_k^-(a)=\theta_{d_v,d_u}$ saturates these bounds and the mappings so defined a Lipschitz continuous trivially. Thus we have
		\begin{subequations}
			\begin{align}
			(\textbf{c}_k^0)^T\textbf{x}_k^0&= \left(m_u\lor m_v-m_u\land m_v\right)\label{equation: ProofTA}\\
			(\textbf{c}_k^-)^T\textbf{x}_k^-&=\left[m_v-m_u\right]_+\label{equation: ProofTB}
			\end{align}
		\end{subequations}
		whenever $|R_k\cap (N(u)\cup N(v))|=1$.
		
		If $|R_k\cap (N(u)\cup N(v))|=2$ we have seven distinct possibilities which we shall treat separately below. Note that $\ominus\in \set{0,-}$ and $a$ and $b$ denote the two distinct vertices in $R_k\cap (N(u)\cup N(v))$ throughout:
		\begin{enumerate}
			\item $a,\:b\in \triangle^k(u,v)$. Since $a\neq b$, $\rho(a,b)\geq 1$. For this situation $(\textbf{c}_k^\ominus)^T\textbf{x}^\ominus=(m_u-m_v)(x_k^\ominus(a)+x_k^\ominus(b))$, so we have bounds $(\textbf{c}_k^0)^T\textbf{x}^0 \leq  2\left(m_u\lor m_v-m_u\land m_v\right)$ and $(\textbf{c}_k^{-})^T\textbf{x}^{-}\leq 2\left[m_v-m_u\right]_+$. These bounds are realised by the (Lipschitz) assignation $x_k^0(a)=x_k^0(b)=2\theta_{m_u,m_v}-1$ and $x_k^-(a)=x_k^-(b)=\theta_{m_u,m_v}-1$. Thus
			\begin{subequations}
				\begin{align}
				(\textbf{c}_k^0)^T\textbf{x}_k^0&= 2\left(m_u\lor m_v-m_u\land m_v\right)\label{equation: ProofTTA}\\
				(\textbf{c}_k^-)^T\textbf{x}_k^-&=2\left[m_v-m_u\right]_+.\label{equation: ProofTTB}
				\end{align}
			\end{subequations}
			\item $a\in \triangle^k(u,v)$ and $b\in \square(u)$. Any such configuration requires $ab\in E(G)$. In this situation we have $(\textbf{c}_k^\ominus)^T\textbf{x}^\ominus=\left(m_u-m_v\right)x_k^\ominus(a)+x_k^\ominus(b)m_u$. If $m_u\geq m_v$, we have the bounds $(\textbf{c}_k^0)^T\textbf{x}^0 \leq \left(m_u-m_v\right) +m_u$ and $(\textbf{c}_k^-)^T\textbf{x}^-\leq m_u$, while if $m_v>m_u$ then we have $(\textbf{c}_k^\ominus)^T\textbf{x}^\ominus \leq m_u\lor \left(m_v-m_u\right)$. Thus if we assign 
			\begin{align}
			x_k^0(a)&=\theta_{m_u,m_v} +(1-\theta_{m_u,m_v})(\theta_{m_u,m_v-m_u}-1) & x_k^0(b)&=\theta_{m_u,m_v} +(1-\theta_{m_u,m_v})\theta_{m_u,m_v-m_u}\nonumber\\
			x_k^-(a)&=\theta_{m_u,m_v}-1+(1-\theta_{m_u,m_v})\theta_{m_u,m_v-m_u} & x_k^-(b)&=\theta_{m_u,m_v} +(1-\theta_{m_u,m_v})\theta_{m_u,m_v-m_u}\nonumber
			\end{align}
			we have the required saturation of bounds by Lipschitz functions. Now noting that
			\begin{align}
			m_u\lor(m_v-m_u)&=(m_v-m_u)+m_u-m_u\land (m_v-m_u) & m_u\lor m_v&=m_u-m_v+m_u\land m_v \nonumber
			\end{align}
			we have
			\begin{subequations}
				\begin{align}
				(\textbf{c}_k^0)^T\textbf{x}_k^0&= \left(m_u\lor m_v-m_u\land m_v\right)+m_u-[m_u\land (m_v-m_u)]_+\label{equation: ProofTS_uA}\\
				(\textbf{c}_k^-)^T\textbf{x}_k^-&=m_u+[m_v-m_u]_+-\left[m_u\land (m_v-m_u)\right]_+\label{equation: ProofTS_uB}
				\end{align}
			\end{subequations}
			\item $a\in \triangle^k(u,v)$ and $b\in \square^k(v)$. Then $ab\in E(G)$ and $(\textbf{c}_k^\ominus)^T\textbf{x}^\ominus=\left(m_u-m_v\right)x_k^\ominus(a)-x_k^\ominus(b)m_v$. If $m_v>m_u$ we have $(\textbf{c}_k^0)^T\textbf{x}^0 \leq \left(m_v-m_u\right) + m_v$ and $(\textbf{c}_k^-)^T\textbf{x}^-\leq \left(m_v-m_u\right) + 2m_v\nonumber$, while $m_u\geq m_v$ implies:
			\begin{align}
			(\textbf{c}_k^0)^T\textbf{x}^0 &\leq (m_u-m_v)\lor \left(m_v\right)\lor (2m_v-(m_u-m_v))=(m_u-m_v)\lor (3m_v-m_u)\nonumber\nonumber\\
			(\textbf{c}_k^1)^T\textbf{x}^1&\leq m_v\lor (2m_v-(m_u-m_v))=m_v\lor (3m_v-m_u).\nonumber
			\end{align}
			The assignments
			\begin{align}
			x_k^0(a)&=(\theta_{m_u,m_v}-1)+\theta_{m_u,m_v}(1-2\theta_{m_v,m_u-m_v}) & x_k^0(b)&=(\theta_{m_u,m_v}-1)-2\theta_{m_u,m_v}\theta_{m_v,m_u-m_v}\nonumber\\
			x_k^-(a)&=(\theta_{m_u,m_v}-1)-\theta_{m_u,m_v}\theta_{m_v,m_u-m_v} & x_k^-(b)&=2(\theta_{m_u,m_v}-1)-\theta_{m_u,m_v}(1+\theta_{m_v,m_u-m_v})\nonumber
			\end{align}
			then give optimal feasible solutions to the relevant linear programmes. Also since
			\begin{align}
			(m_u-m_v)\lor (3m_v-m_u)&=(m_u-m_v)+m_v-m_v\land (2m_u-3m_v)\nonumber\\
			m_v\lor (3m_v-m_u)&=2m_v+[m_v-m_u]_+-m_v\land (m_u-m_v)\nonumber
			\end{align}
			we have
			\begin{subequations}
				\begin{align}
				(\textbf{c}_k^0)^T\textbf{x}_k^0&= \left(m_u\lor m_v-m_u\land m_v\right)+m_v-[m_v\land (2m_u-3m_v)]_+\label{equation: ProofTS_vA}\\
				(\textbf{c}_k^-)^T\textbf{x}_k^-&=2m_v+[m_v-m_u]_+-[m_v\land (m_u-m_v)]_+\label{equation: ProofTS_vB}
				\end{align}
			\end{subequations}
			\item $a\in \square^k(u)$ and $b\in \square^k(v)$. Then $a\in N(b)$ and $(\textbf{c}_k^\ominus)^T\textbf{x}^\ominus=x_k^\ominus(a)m_u-x_k^\ominus(b)m_v$, $(\textbf{c}_k^0)^T\textbf{x}^0\leq m_u\lor m_v$ and $(\textbf{c}_k^-)^T\textbf{x}^-\leq m_u\lor m_v\lor (2m_v-m_u)=m_u\lor (2m_v-m_u)$. Then the assignment
			\begin{align}
			x_k^0(a)&=\theta_{m_u,m_v} & x_k^0(b)&=\theta_{m_u,m_v}-1 & x_k^-(a)&=2\theta_{m_u,m_v}-1 & x_k^-(b)&=-2(1-\theta_{m_u,m_v})\nonumber
			\end{align}
			gives the required optimal Lipschitz function. Clearly
			\begin{subequations}
				\begin{align}
				(\textbf{c}_k^0)^T\textbf{x}_k^0&= m_u\lor m_v=m_u+m_v-m_u\land m_v\label{equation: ProofSSA}\\
				(\textbf{c}_k^-)^T\textbf{x}_k^-&= m_u+2m_v-2m_u\land m_v\label{equation: ProofSSB}
				\end{align}
			\end{subequations}
			\item $a\in \triangle^k(u,v)$ and $b\in \pentagon^k(u)$. Note that this case is equivalent to the case $|R_k\cap (N(u)\cup N(v))|=1$ for $W_0$ so we only consider the case of $W_{-1}$. Clearly $\rho(a,b)=2$ and $(\textbf{c}_k^-)^T\textbf{x}^-\leq \left(m_u-m_v\right)x_k^-(a)+x_k^-(b)m_u.$ If $m_u\geq m_v$ then $(\textbf{c}_k^-)^T\textbf{x}^-\leq m_u$ while if $m_v>m_u$ we have $(\textbf{c}_k^-)^T\textbf{x}^-\leq m_u+(m_v-m_u)=m_v$, i.e. 
			\begin{align}
			(\textbf{c}_k^-)^T\textbf{x}^-_k=m_u+[m_v-m_u]_+
			\end{align}
			where it is sufficient to specify $x_k^-(a)=\theta_{m_u,m_v}-1$ and $x_k^-(b)=1$.
			\item $a\in \triangle^k(u,v)$ and $b\in \pentagon^k(v)$. Again we need only consider the case $W_{-1}$, $\rho(a,b)=2$ and $(\textbf{c}_k^-)^T\textbf{x}^-\leq\left(m_u-m_v\right)x_k^-(a)-x_k^-(b)m_v.$ From this it is immediately apparent that $(\textbf{c}_k^-)^T\textbf{x}^-\leq 2m_v+[m_v-m_u]_+$ so choosing $x_k^-(a)=\theta_{m_u,m_v}-1$ and $x_k^-(b)=-2$ implies
			\begin{align}
			(\textbf{c}_k^-)^T\textbf{x}^-_k= 2m_v+[m_v-m_u]_+.
			\end{align}
			\item $a\in \pentagon^k(u)$ and $b\in \pentagon^k(v)$. We ignore the case $W_0$ and note that $\rho(a,b)=2$ and $(\textbf{c}_k^-)^T\textbf{x}^-\leq x_k^-(a)m_u-x_k^-(b)m_v$. Clearly, $(\textbf{c}_k^-)^T\textbf{x}^-\leq (m_u+m_v)\lor (2m_v)$ which is immediately realised by a feasible solution if we assign $x_k^-(a)=\theta_{m_u,m_v}$ and $x_k^b(a)=\theta_{m_u,m_v}-2$. Then
			\begin{align}
			(\textbf{c}_k^-)^T\textbf{x}^-_k= (m_u+m_v)\lor (2m_v)=2m_v+m_u-m_u\land m_v
			\end{align}
		\end{enumerate}
		We can now write down expressions for $W_0$ and $W_1$. We consider $W_0$ first:  considering equations \ref{equation: ProofTA}, \ref{equation: ProofTTA}, \ref{equation: ProofTS_uA}, \ref{equation: ProofTS_vA} and \ref{equation: ProofSSA} in conjunction immediately implies:
		\begin{align}
		W_0(u,v)&=m_u(n_u+\pentagon_u)+m_v(n_v+\pentagon_v)+(m_u\lor m_v-m_u\land m_v)\triangle_{uv}+m_u\square_u+m_v\square_v\nonumber\\
		&\qquad -\square_\triangle^u[m_u\land (m_v-m_u)]_+-\square_\triangle^v[m_v\land (2m_u-3m_v)]_+-m_u\land m_v\square_\circ\nonumber\\
		&=(1-m_u\land m_v \triangle_{uv})+(1-m_u-m_v-m_u\lor m_v\triangle_{uv}-(m_u\square_u+m_v\square_v))\nonumber\\
		&\qquad +(m_u\lor m_v-m_u\land m_v)\triangle_{uv}+m_u\square_u+m_v\square_v -\square_\triangle^u[m_u\land (m_v-m_u)]_+-\square_\triangle^v[m_v\land (2m_u-3m_v)]_+\nonumber\\
		&=(1-m_u\land m_v \triangle_{uv})+(1-m_u-m_v-m_u\land m_v\triangle_{uv}-m_u\land m_v\square_\circ)-\square_\triangle^u[m_u\land (m_v-m_u)]_+\nonumber\\
		&\qquad -\square_\triangle^v[m_v\land (2m_u-3m_v)]_+.\nonumber
		\end{align}
		Similarly:
		\begin{align}
		W_{-1}(u,v)&=m_u(n_u-1)+2m_vn_v+[m_v-m_u]_+\triangle_{uv}+m_u\square_u+2m_v\square_v +m_u\pentagon_u+2m_v\pentagon_v\nonumber\\
		&\qquad -\square_\triangle^u[m_u\land(m_v-m_u)]_+\nonumber-\square_\triangle^v[m_v\land (m_u-m_v)]_+-2m_u\land m_v\square_\circ-\pentagon_{\circ}m_u\land m_v\nonumber\\
		&=(1-m_u\land m_v\triangle_{uv})+(1-m_u-m_v-m_v\triangle_{uv}-(m_u\square_u +m_v\square_v))\nonumber\\
		&\qquad \qquad +(1-m_u-m_v-m_u\lor m_v\triangle_{uv}-m_v\square_v-(m_u\pentagon_u+2m_v\pentagon_v))\nonumber\\
		&\qquad +[m_v-m_u]_+\triangle_{uv}+m_u\square_u+2m_v\square_v +m_u\pentagon_u+2m_v\pentagon_v\nonumber\\
		&\qquad -\square_\triangle^u[m_u\land(m_v-m_u)]_+\nonumber-\square_\triangle^v[m_v\land (m_u-m_v)]-2m_u\land m_v\square_\circ-\pentagon_{\circ}m_u\land m_v\nonumber\\
		&=(1-m_u\land m_v\triangle_{uv})+(1-m_u-m_v-m_u\land m_v\triangle_{uv}-m_u\land m_v \square_\circ)\nonumber\\
		&\qquad \qquad +(1-m_u-m_v-m_u\lor m_v\triangle_{uv}-m_u\land m_v \square_\circ-m_u\land m_v\pentagon_{\circ})\nonumber\nonumber\\
		&\qquad -\square_\triangle^u[m_u\land(m_v-m_u)]_+-\square_\triangle^v[m_v\land (m_u-m_v)]_+.\nonumber
		\end{align}
		This immediately implies
		\begin{align}
		W(u,v)&=(1-m_u\land m_v\triangle_{uv})\nonumber\\
		&\qquad +\left[1-m_u-m_v-m_u\land m_v\triangle_{uv}-m_u\land m_v \square_\circ -m_u\land [m_v-m_u]_+\square_\triangle^u -m_v\land [2m_u-3m_v]_+\square_\triangle^v\right]_+\nonumber\\
		&\qquad +\left[1-m_u-m_v-m_u\lor m_v\triangle_{uv}-m_u\land m_v (\square_\circ+\pentagon_{\circ})\nonumber -(m_v\land [m_u-m_v]_+-m_v\land [2m_u-3m_v]_+)\square_\triangle^v\right]_+\nonumber
		\end{align}
		as required.
	\end{proof}
	\section*{Acknowledgements}
	The author would like to acknowledge studentship funding from EPSRC under the grant number EP/L015110/1. Also, many thanks to Fabio Biancalana and Carlo Trugenberger for helpful discussion on the topic and advice about the text. David Cushing pointed out an error in counterexample \ref{counterexample1} in an earlier version of this paper, which was corrected with the help of the graph curvature calculator \cite{GraphCurvatureCalculator}.
	\section*{References}
	\bibliographystyle{plain}
	\bibliography{Ref}
\end{document}